\documentclass[twoside, 11pt]{article}
\usepackage{amssymb, amsmath, mathrsfs, amsthm}
\usepackage{graphicx}
\usepackage{color}
\usepackage[top=2cm, bottom=2cm, left=2cm, right=2cm]{geometry}
\usepackage{float, caption, subcaption}

\input{epsf}

\newcommand{\bd}{\begin{description}}
\newcommand{\ed}{\end{description}}
\newcommand{\bi}{\begin{itemize}}
\newcommand{\ei}{\end{itemize}}
\newcommand{\be}{\begin{enumerate}}
\newcommand{\ee}{\end{enumerate}}
\newcommand{\beq}{\begin{equation}}
\newcommand{\eeq}{\end{equation}}
\newcommand{\beqs}{\begin{eqnarray*}}
\newcommand{\eeqs}{\end{eqnarray*}}

\definecolor{DarkGreen}{rgb}{0.2, 0.6, 0.3}

\newtheorem{theorem}{Theorem}[section]

\newtheorem{lemma}{Lemma}[section]

\newtheorem{corollary}[theorem]{Corollary}

\newtheorem{claim}{Claim}
\newtheorem{remark}{Remark}[section]
\newtheorem{fact}{Fact}
\newtheorem{proposition}{Proposition}[section]

\newtheorem{observation}{Observation}[section]
\begin{document}
\title{\textbf{On the $g$-good-neighbor connectivity of graphs} \footnote{Supported by the National
Science Foundation of China (Nos. 11601254, 11551001, 11161037,
61763041, 11661068, and 11461054) and the Science Found of Qinghai
Province (Nos.  2016-ZJ-948Q, and 2014-ZJ-907) and the  Qinghai Key
Laboratory of Internet of Things Project (2017-ZJ-Y21).} }

\author{
Zhao Wang\footnote{College of Science, China Jiliang University,
Hangzhou 310018, China. {\tt wangzhao@mail.bnu.edu.cn}}, \ \ Yaping
Mao\footnote{Corresponding author} \footnote{School of Mathematics
and Statistis, Qinghai Normal University, Xining, Qinghai 810008,
China. {\tt maoyaping@ymail.com}}, \ \ Sun-Yuan Hsieh
\footnote{Department of Computer Science and Information
Engineering, National Cheng Kung University, Tainan 701, Taiwan {\tt
hsiehsy@mail.ncku.edu.tw}}, \ \ Jichang Wu \footnote{School of
Mathematics, Shandong University, Jinan, Shandong 250100, China {\tt
jichangwu@yahoo.com.cn}}}
\date{}
\maketitle

\begin{abstract}
Connectivity and diagnosability are two important parameters for the
fault tolerant of an interconnection network $G$. In 1996,
F\`{a}brega and Fiol proposed the $g$-good-neighbor connectivity of
$G$. In this paper, we show that $1\leq \kappa^g(G)\leq n-2g-2$ for
$0\leq g\leq \left\{\Delta(G),\left\lfloor
\frac{n-3}{2}\right\rfloor\right\}$, and graphs with
$\kappa^g(G)=1,2$ and trees with $\kappa^g(T_n)=n-t$ for $4\leq
t\leq \frac{n+2}{2}$ are characterized, respectively. In the end, we
get the three extremal
results for the $g$-good-neighbor connectivity. \\[2mm]
\end{abstract}

\section{Introduction}

With the rapid development of VLSI technology, a multiprocessor
system may contain hundreds or even thousands of nodes, and some of
them may be faulty when the system is implemented. As the number of
processors in a system increases, the possibility that its
processors may be comefaulty also increases. Because designing such
systems without defects is nearly impossible, reliability and fault
tolerance are two of the most critical concerns of multiprocessor
systems \cite{XuWangWang}.

By the definition proposed by Esfahanian \cite{Esfahanian}, a
multiprocessor system is fault tolerant if it can remain functional
in the presence of failures. Two basic functionality criteria have
received considerable attention. The first criterion for a system to
be regarded as functional is whether the network logically contains
a certain topological structure. This is the problem that occurs
when embedding one architecture into another \cite{Leighton, Xu}.
This approach involves using system-wide redundancy and
reconfiguration. The second functionality criterion considers a
multiprocessor system functional if a fault-free communication path
exists between any two fault-free nodes; that is, the topological
structure of the multiprocessor system remains connected in the
presence of certain failures. Thus, connectivity and edge
connectivity are two major measurements of this criterion \cite{Xu}.
The \emph{connectivity} of a graph $G$, denoted by $\kappa(G)$, is
the minimal number of vertices whose removal from produces a
disconnected graph or only one vertex; the \emph{edge connectivity}
of a graph $G$, denoted by $\lambda(G)$, is the minimal number of
edges whose removal from produces a disconnected graph. However,
these two parameters tacitly assume that all vertices that are
adjacent to, or all edges that are incident to, the same vertex can
potentially fail simultaneously. This is practically impossible in
some network applications.

For a graph $G$, let $V(G)$, $E(G)$, $e(G)$, $\overline{G}$, and
$diam(G)$ denote the set of vertices, the set of edges, the size,
the complement, and the diameter of $G$, respectively. A subgraph
$H$ of $G$ is a graph with $V(H)\subseteq V(G)$, $E(H)\subseteq
E(G)$, and the endpoints of every edge in $E(H)$ belonging to
$V(H)$. For any subset $X$ of $V(G)$, let $G[X]$ denote the subgraph
induced by $X$; similarly, for any subset $F$ of $E(G)$, let $G[F]$
denote the subgraph induced by $F$. We use $G-X$ to denote the
subgraph of $G$ obtained by removing all the vertices of $X$
together with the edges incident with them from $G$; similarly, we
use $G-F$ to denote the subgraph of $G$ obtained by removing all the
edges of $F$ from $G$. If $X=\{v\}$ and $F=\{e\}$, we simply write
$G-v$ and $G-e$ for $G-\{v\}$ and $G-\{e\}$, respectively. For two
subsets $X$ and $Y$ of $V(G)$ we denote by $E_G[X,Y]$ the set of
edges of $G$ with one end in $X$ and the other end in $Y$. If
$X=\{x\}$, we simply write $E_G[x,Y]$ for $E_G[\{x\},Y]$. The {\it
degree}\index{degree} of a vertex $v$ in a graph $G$, denoted by
$deg_G(v)$, is the number of edges of $G$ incident with $v$. Let
$\delta(G)$ and $\Delta(G)$ be the minimum degree and maximum degree
of the vertices of $G$, respectively. The set of neighbors of a
vertex $v$ in a graph $G$ is denoted by $N_G(v)$. The {\it union}
$G\cup H$ of two graphs $G$ and $H$ is the graph with vertex set
$V(G)\cup V(H)$ and edge set $E(G)\cup E(H)$. If $G$ is the disjoint
union of $k$ copies of a graph $H$, we simply write $G=kH$.

\subsection{The $g$-extra (edge-)connectivity}

F\`{a}brega and Fiol \cite{FabregaFiol1, FabregaFiol2} introduced
the following measures. A subset of vertices $S$ is said to be a
\emph{cutset} if $G-S$ is not connected. A cutset $S$ is called an
\emph{$R_g$-cutset}, where $g$ is a non-negative integer, if every
component of $G-S$ has at least $g+1$ vertices. If $G$ has at least
one $R_g$-cutset, the \emph{$g$-extra connectivity} of $G$, denoted
by $\kappa_g(G)$, is then defined as the minimum cardinality over
all $R_g$-cutsets of $G$. A connected graph $G$ is said to be
\emph{$g$-extra connected} if $G$ has a $g$-extra cut.

Given a graph and a non-negative integer $g$, the \emph{$g$-extra
edge-connectivity}, written as $\lambda_g(G)$, is the minimal
cardinality of a set of edges in $G$, if it exists, whose deletion
disconnects $G$ and leaves each remaining component with more than
$g$ vertices.

Note that $\kappa_0(G)=\kappa(G)$ and $\lambda_0(G)=\lambda(G)$ for
any graph if is not a complete graph. Therefore, the $g$-extra
connectivity (resp. $g$-extra edge connectivity) can be regarded as
a generalization of the classical connectivity (resp. classical edge
connectivity) that provides measures that are more accurate for
reliability and fault tolerance for large-scale parallel processing
systems. Regarding the computational complexity of the problem,
based on thorough research, no polynomial-time algorithm has been
presented to compute $\kappa_g$ for a general graph; nor has there
been any tight upper bound for $\kappa_g$ \cite{Esfahanian}.
However, $\lambda_1$ can be computed by solving numerous network
flow problems \cite{EsfahanianHakimi}.

\subsection{The $g$-good-neighbor connectivity}

Let $G=(V,E)$ be connected. A fault set $F\subseteq V$ is called a
\emph{$g$-good-neighbor faulty set} if $|N(v)\cap (V-F)|\geq g$ for
every vertex $v$ in $V-F$. A \emph{$g$-good-neighbor cut} of $G$ is
a $g$-good-neighbor faulty set $F$ such that $G-F$ is disconnected.
The minimum cardinality of $g$-good-neighbor cuts is said to be the
\emph{$g$-good-neighbor connectivity} of $G$, denoted by
$\kappa^{g}(G)$. A connected graph $G$ is said to be
\emph{$g$-good-neighbor connected} if $G$ has a $g$-good-neighbor
cut. For more research on $g$-good-neighbor connectivity, we refer
to \cite{LiLu, LinXuWangZhou, PengLinTanHsu, RenWang, WangWang,
WangWangWang, WeiXu, WeiXu2, XuLiZhouHaoGu, YuanLiuQinZhangLi,
YuanLiuMaLiuQinZhang}.

The relation of $g$-extra connectivity and $g$-good-neighbor
connectivity was given in \cite{PreparataMetzeChien}.
\begin{theorem}{\upshape \cite{PreparataMetzeChien}}\label{def1-1}
$(1)$ If $G$ is a $1$-good-neighbor connected graph, then
$\kappa_{1}(G)=\kappa^{1}(G)$.

$(2)$ If $G$ is a $g$-extra and $g$-good-neighbor connected graph,
then $\kappa_{g}(G)\leq \kappa^{g}(G)$.
\end{theorem}

The range of the integer $g$ can be determined immediately.
\begin{proposition}\label{pro1-1}
Let $g$ be a non-negative integer. If $G$ has its $g$-good-neighbor
connectivity, then
$$
0\leq g\leq \left\{\Delta(G),\left\lfloor \frac{n-3}{2}\right\rfloor\right\}
$$
and
$$
e(G)\leq {n\choose 2}-(g+1)^2.
$$
\end{proposition}
\begin{proof}
From the definition of $g$-good-neighbor connectivity, there exists
$X\subseteq V(G)$ with $|X|=\kappa_g(G)$ such that $G-X$ is not
connected and the minimum degree of each component of $G-X$ is at
least $g$. Let $C_1,C_2,\ldots,C_r$ be the components of $G-X$. Then
$r\geq 2$ and
$$
2(g+1)\leq |V(C_1)|+|V(C_2)|\leq \sum_{i=1}^r|V(C_i)|=|V(G)|-|X|\leq
n-1,
$$
and hence $0\leq g\leq \lfloor \frac{n-3}{2}\rfloor$. For each $C_i
\ (1\leq i\leq r)$, we have $\Delta(G)\geq \Delta(C_i)\geq
\delta(C_i)\geq g$.

Furthermore, $e(\overline{G})\geq |V(C_1)||V(C_2)|=(g+1)^2$, and
hence $e(G)\leq {n\choose 2}-(g+1)^2$.
\end{proof}

The monotone property of $\kappa_g(G)$ for non-negative integer $g$
is true.
\begin{proposition}\label{pro1-2}
Let $g$ be a non-negative integer, and let $G$ be a connected graph.
Then
$$
\kappa^{g}(G)\leq \kappa^{g+1}(G).
$$
\end{proposition}
\begin{proof}
From the definition of $(g+1)$-good-neighbor connectivity, by
deleting $\kappa^{g+1}(G)$ vertices in $G$, the resulting graph is
not connected and the minimum degree of each component is at least
$g+1>g$, and hence $\kappa^{g}(G)\leq \kappa^{g+1}(G)$.
\end{proof}

The monotone property of $\kappa^0(G)$ is true in terms of connected
graphs $G$.
\begin{observation}\label{obs1-2}
Let $G$ be a connected graph. If $H$ is a spanning subgraph of $G$,
then $\kappa^0(H)\leq \kappa^0(G)$.
\end{observation}

But for $g\geq 1$, the above monotone property is not true.

\begin{remark}\label{rem1-1}
Let $G$ be a graph obtained from four cliques $X_1,X_2,Y_1,Y_2$ with
$\delta(X_i)\geq g \ (i=1,2)$ and $\delta(Y_j)\geq g \ (j=1,2)$ and
three vertices $u,v,w$ by adding edges in $E_G[u,X_1]\cup
E_G[u,X_2]\cup E_G[u,Y_1]\cup E_G[u,Y_2]\cup E_G[v,Y_1]\cup
E_G[v,Y_2]\cup E_G[w,Y_1]\cup E_G[w,Y_2]\cup \{uv,uw\}$. Let $H$ be
a graph obtained from two cliques $Y_1,Y_2$ defined in $G$, two
subgraph $Z_1,Z_2$ such that $\delta(Z_i)=g-1 \ (i=1,2)$, and three
vertices $u,v,w$ by adding edges in $E_G[u,Z_1]\cup E_G[u,Z_2]\cup
E_G[u,Y_2]\cup E_G[v,Y_1]\cup E_G[v,Y_2]\cup E_G[w,Y_1]\cup
E_G[w,Y_2]$. Clearly, $H$ is a spanning subgraph of $G$; see Figure
$1$. We first show that $\kappa^g(G)=1$. By deleting the vertex $u$,
there are three components and the minimum degree of each component
is at least $g$, and hence $\kappa^g(G)\leq 1$, and so
$\kappa^g(G)=1$. Next, we show that $\kappa^g(H)=2$. Suppose
$\kappa^g(H)=1$. Then there exists a vertex such that by deleting
this vertex the resulting graph is not connected and the minimum
degree of each component is at least $g$. Note that $u$ is the
unique cut vertex of $H$. Clearly, there is a component having
minimum degree $g-1$, a contradiction. So $\kappa^g(H)\geq 2$. By
deleting the vertices $v,w$, there are two components and the
minimum degree of each of them is at least $g$, and hence
$\kappa^g(H)\leq 2$. So $\kappa^g(H)=2>\kappa^g(G)$.
\end{remark}
\begin{figure}[!hbpt]
\begin{center}
\includegraphics[scale=0.7]{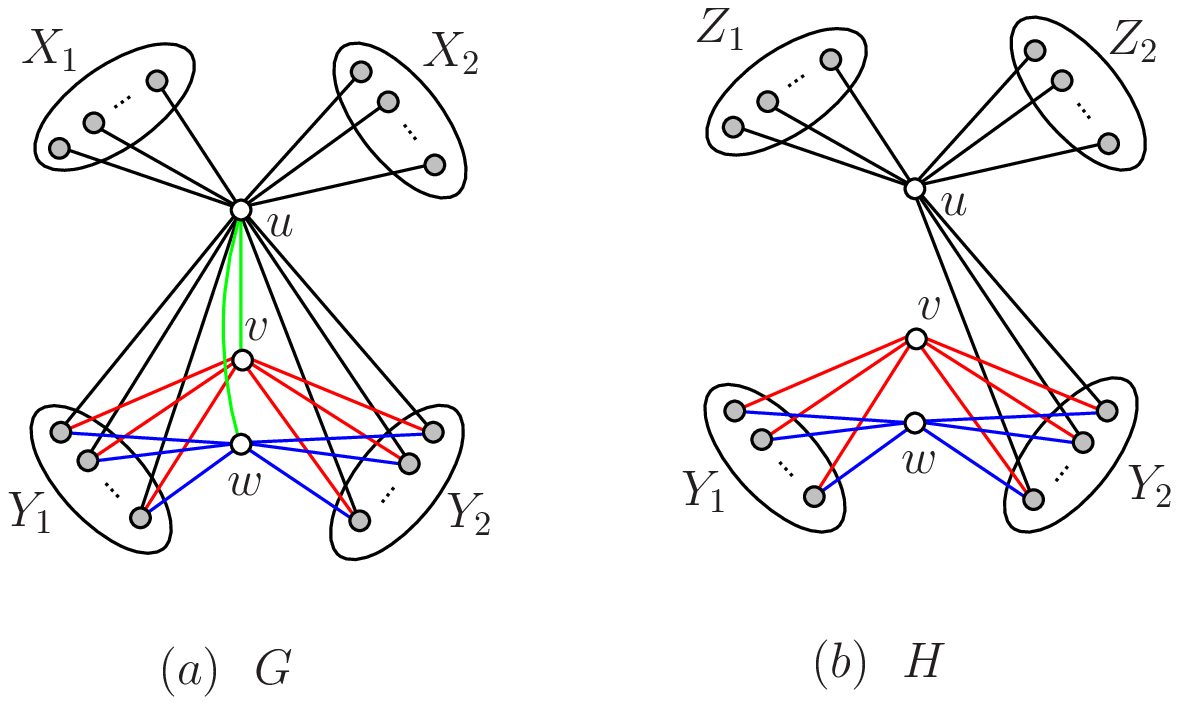}
\end{center}
\begin{center}
Figure 1: Graphs for Remark 1.1.
\end{center}\label{fig1}
\end{figure}

\subsection{Some classical problems}

One of the interesting problems in extremal graph theory is the
Erd\"{o}s-Gallai-type problem, which is to determine the maximum or
minimum value of a graph parameter with some given properties. In
\cite{CaiLiWu, JiangLiZhang}, the authors investigated two kinds of
Erd\"{o}s-Gallai-type problems for monochromatic connection number
and monochromatic vertex connection number, respectively. Motivated
by these, we study two kinds of Erd\"{o}s-Gallai-type problems for
$\kappa^g(G)$ in this paper.

\noindent {\bf Problem 1.} Given two positive integers $n$ and $k$,
compute the minimum integer $f(n,k)$ such that for every connected
graph $G$ of order $n$, if $|E(G)|\geq f(n,k)$ then $\kappa^g(G)\geq
k$.

\noindent {\bf Problem 2.} Given two positive integers $n$ and $k$,
compute the maximum integer $g(n,k)$ such that for every graph $G$
of order $n$, if $|E(G)|\leq g(n,k)$ then $\kappa^g(G)\leq k$.

Another interesting problem in extremal graph theory is to study the
minimum size of graphs with given parameter; see \cite{Schiermeyer}.

\noindent {\bf Problem 3.} Given two positive integers $n$ and $k$,
compute the minimum integer $s(n,k)=\min\{|E(G)|:G\in
\mathscr{G}(n,k)\}$, where $\mathscr{G}(n,k)$ the set of all graphs
of order $n$ (that is, with $n$ vertices) with $g$-good-neighbor
connectivity $k$.

In Section $2$, we obtain the exact values of $g$-extra
connectivities of complete bipartite graphs, complete multipartite
graphs, wheels and paths. For a connected graph $G$ of order $n$, we
show that $\kappa(G)\leq \kappa^g(G)\leq n-2g-2$ for $0\leq g\leq
\min\{\Delta(G),\left\lfloor \frac{n-3}{2}\right\rfloor\}$, and
$1\leq \kappa^g(G)\leq n-2g-2$ for $0\leq g\leq \left\lfloor
\frac{n-3}{2}\right\rfloor$ in Section $2$. Graphs with
$\kappa_g(G)=1,2$ and trees with $\kappa^g(T_n)=n-t$ are
characterized, respectively, in Section $3$. In the end, we get the
extremal results for the $g$-good neighbor connectivity in Section
$4$.

\section{Results for special graphs}

The following upper and lower bounds are immediate.
\begin{proposition}\label{pro3-1}
Let $G$ be a connected graph of order $n$, and let $g$ be a
non-negative integer such that $0\leq g\leq
\min\{\Delta(G),\left\lfloor \frac{n-3}{2}\right\rfloor\}$. Then
$$
\kappa(G)\leq \kappa^g(G)\leq n-2g-2.
$$
Moreover, the upper and lower bounds are sharp.
\end{proposition}
\begin{proof}
From the definition of $\kappa^g(G)$, we have $\kappa^g(G)\geq
\kappa(G)$. Suppose $\kappa^g(G)\geq n-2g-1$. From the definition,
we can delete $\kappa^g(G)$ vertices in $G$ such that there are at
least two components and one of them has no more than $g$ vertices,
a contradiction. So $\kappa(G)\leq \kappa^g(G)\leq n-2g-2$. Theorem
\ref{th4-3} shows that the upper bound is sharp. If $k=0$, then
$\kappa(G)=\kappa^0(G)$. This implies that the lower bound is sharp.
\end{proof}

The following corollary is immediate from Proposition \ref{pro3-1}.
\begin{corollary}\label{cor3-1}
Let $n,g$ be two integers with $0\leq g\leq \left\lfloor
\frac{n-3}{2}\right\rfloor$. If $G$ is a connected graph of order
$n$, then
$$
1\leq \kappa^g(G)\leq n-2g-2.
$$
Moreover, the upper and lower bounds are sharp.
\end{corollary}

In the following, we obtain the exact values for $g$-good neighbor
connectivity of some special graphs.
\begin{proposition}\label{pro2-2}
Let $g$ be a non-negative integer.

$(1)$ If $K_{a,b} \ (a\geq b\geq 2)$ is a complete bipartite graph,
then $g=0$ and $\kappa^0(K_{a,b})=b$ and $\kappa^g(K_{a,b})$ does
not exist for $g\geq 1$.

$(2)$ Let $r$ be an integer with $r\geq 3$. For complete
multipartite graph $K_{n_1,n_2,\ldots,n_r}$ $(n_1\leq n_2\leq \ldots
\leq n_r)$, we have $g=0$ and
$$
\kappa^0(K_{n_1,n_2,\ldots,n_r})=\sum_{i=1}^{r-1}n_i
$$
and $\kappa^g(K_{n_1,n_2,\ldots,n_r})$ does not exist for $g\geq 1$.
\end{proposition}
\begin{proof}
$(1)$ By deleting any vertex in $K_{a,b}$, the resulting graph is
still a complete bipartite graph and it is connected. If we require
the resulting graph is not connected, then we must delete all the
vertices of one part. Then $g=0$. Since $a\geq b\geq 2$, we have
$\kappa^g(K_{a,b})=b$.

$(2)$ Similarly to the proof of $(1)$, we can get
$\kappa^g(K_{n_1,n_2,\ldots,n_r})=\sum_{i=1}^{r-1}n_i$.
\end{proof}

\begin{proposition}\label{pro2-3}
Let $g$ be a non-negative integer.

$(1)$ If $W_{n} \ (n\geq 7)$ is a wheel of order $n$, then
$\kappa^g(W_{n})=3$ for $g=0,1$.

$(2)$ If $P_{n} \ (n\geq 3)$ be a path of order $n$, then
$\kappa^g(P_{n})=1$ for $g=0,1$.
\end{proposition}
\begin{proof}
$(1)$ From the definition of $\kappa^g(W_{n})$, there exists
$X\subseteq V(W_{n})$ such that $W_{n}-X$ is not connected and the
minimum degree of each component of $W_{n}-X$ is at least $g$. Note
that each component is a path. Then $g\leq 1$. From Proposition
\ref{pro3-1}, we have $\kappa^g(W_{n})\geq \kappa(W_{n})=3$. It
suffices to show that $\kappa^g(W_{n})\leq 3$ for $g=0,1$. Let $v$
be the center of $W_{n}$, and $W_{n}-v=C_{n-1}$, and
$V(C_{n-1})=\{u_1,u_2,\ldots,u_{n-1}\}$. Choose
$X=\{v,u_1,u_{\lceil\frac{n}{2}\rceil}\}$. Since $g=0,1$, it follows
that the minimum degree of each component of $W_{n}-X$ is at least
$g$, and hence $\kappa^g(W_{n})\leq 3$. So $\kappa^g(W_{n})=3$.

$(2)$ Similarly to the proof of $(1)$, we have $g=0,1$. From
Proposition \ref{pro3-1}, we have $\kappa^g(P_{n})\geq
\kappa(P_{n})=1$. It suffices to show $\kappa^g(P_{n})\leq 1$. Let
$P_n=u_1u_2\ldots u_{n}$. Choose $v=u_{\lceil n/2\rceil}$. Since
$0\leq g\leq \lfloor \frac{n-1}{2}\rfloor-1$, it follows that the
minimum degree of each component of $G-v$ is at least $g$, and hence
$\kappa^g(P_{n})\leq 1$. So $\kappa^g(P_{n})=1$.
\end{proof}

\section{Graphs with given $g$-good-neighbor connectivity}

In this section, we first characterize trees with given
$g$-good-neighbor connectivity. Next, we characterize graphs with
small $g$-good-neighbor connectivity.

\subsection{Trees with given $g$-good-neighbor connectivity}

Let $K_{1,n-t-1}$ be a star with center $v$ and leaves
$u_1,u_2,\ldots, u_{n-t-1}$, and let
$K_{1,a_1-1},K_{1,a_2-1},\ldots,K_{1,a_r-1}$ be $r$ stars with
centers $w_1,w_2,\ldots,w_{r}$, respectively. Furthermore, let
$T_n^*$ be a tree of order $n$ obtained from $K_{1,n-t-1}$ and
$K_{1,a_1-1},K_{1,a_2-1},\ldots,K_{1,a_r-1}$ by adding the edges
$\{vw_1,vw_2,\ldots,vw_{r}\}$, where $r\geq 2$, $\sum_{i=1}^ra_i=t$,
and $a_i\geq 2$ for each $i \ (1\leq i\leq r)$; see Figure 1.
\begin{figure}[!hbpt]
\begin{center}
\includegraphics[scale=0.7]{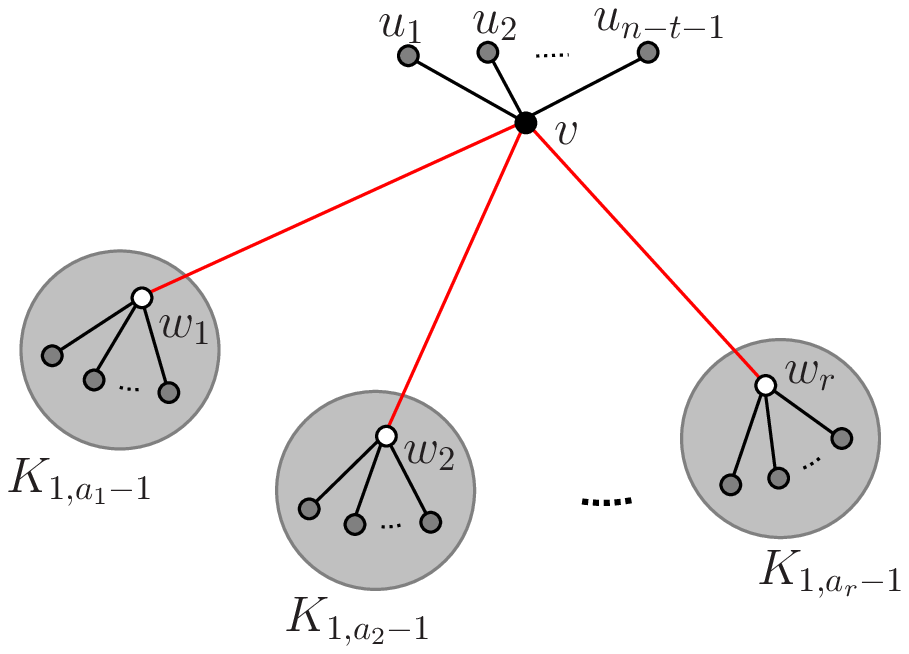}
\end{center}
\begin{center}
Figure 1: Tree $T_n^*$.
\end{center}\label{fig1}
\end{figure}

\begin{lemma}\label{lem4-1}
For $4\leq t\leq \frac{n+2}{2}$, we have
$$
\kappa^1(T_n^*)=n-t.
$$
\end{lemma}
\begin{proof}
Choose $X=\{u_i\,|\,1\leq i\leq n-t-1\}\cup \{v\}$. Then $T_n^*-X$
is not connected and the minimum degree of each component is exactly
$1$. So $\kappa^1(T_n^*)\leq |X|=n-t$. It suffices to show
$\kappa^1(T_n^*)\geq n-t$. It suffices to prove that for any
$X\subseteq V(T_n^*)$ and $|X|\leq n-t-1$, if $T_n^*-X$ is not
connected, then there exists a component of $T_n^*-X$ such that its
minimum degree is exactly $0$. If $v\in X$, then there is an
isolated vertex in the resulting graph, as desired.

Suppose $v\notin X$. Since $T_n^*-X$ is not connected, it follows
that there exits some $w_{j_1}$ such that $w_{j_1}\in X$, and hence
$V(K_{1,a_{j_1}-1})\subseteq X$. Then there exits some $w_{j_2}$
such that $w_{j_2}\in X-w_{j_1}$, and hence
$V(K_{1,a_{j_2}-1})\subseteq X$. Furthermore, there exits some
$w_{j_3}$ such that $w_{j_3}\in X-w_{j_1}-w_{j_2}$, and hence
$V(K_{1,a_{j_3}-1})\subseteq X$. Continue this process, we have
$$
\{w_{j_1},w_{j_2},\ldots,w_{j_r}\}\cup
\left(\bigcup_{i=1}^rV(K_{1,a_{j_i}-1})\right)=\{w_{1},w_{2},\ldots,w_{r}\}\cup
\left(\bigcup_{i=1}^rV(K_{1,a_{i}-1})\right)\subseteq X.
$$
Clearly, $T_n^*-X$ is connected, a contradiction. So
$\kappa^1(T_n^*)\geq n-t$, and hence $\kappa^1(T_n^*)=n-t$.
\end{proof}

Trees with $\kappa^g(T_n)=n-t$ for general $g$ and $t$ can be
characterized.
\begin{theorem}\label{th4-3}
Let $n,g$ be two integers and $T_n$ be a tree of order $n$ with
$0\leq g\leq \left\{\Delta(T_n),\left\lfloor
\frac{n-3}{2}\right\rfloor\right\}$. Then $\kappa^g(T_n)=n-t$ if and
only if $T_n$ satisfies one of the following conditions.

$(1)$ $g=0$ and $n=t+1$;

$(2)$ $g=1$, $4\leq t\leq \frac{n+2}{2}$ and $T_n=T_n^*$.
\end{theorem}
\begin{proof}
If $g=0$ and $n=t+1$, then $\kappa^0(T_{t+1})=1=(t+1)-t=n-t$.
Suppose $g=1$, $4\leq t\leq \frac{n+2}{2}$ and $T_n=T_n^*$. From
Lemma \ref{lem4-1}, we have $\kappa^1(T_n^*)=n-t$.

Conversely, we suppose $\kappa^g(T_n)=n-t$. Then we have the
following claim.

\begin{claim}\label{Clm:1}
$g\leq 1$.
\end{claim}
\begin{proof}
Assume, to the contrary, that $g\geq 2$. From the definition of
$\kappa^g(T_n)$, there exists $X\subseteq V(T_n)$ and
$|X|=\kappa^g(T_n)$ such that $T_n-X$ is not connected and the
minimum degree of each component of $T_n-X$ is at least $g\geq 2$.
Since $T_n$ is a tree, it follows that each component of $T_n-X$ is
a subtree of $T_n$, and the minimum degree of each component is at
most $1$, a contradiction.
\end{proof}

From Claim \ref{Clm:1}, we have $g\leq 1$. If $g=0$, then
$1=\kappa(T_n)=\kappa^0(T_n)=n-t$, and hence $n=t+1$. If $g=1$, then
$\kappa^1(T_n)=n-t$. Then there exists $X\subseteq V(T_n)$ and
$|X|=n-t$ such that $T_n-X$ is not connected and the minimum degree
of each component is exactly $1$. Clearly, there exits a cut vertex
$v$ in $T_n$ such that $v\in X$. Let $C_1,C_2,\ldots,C_r$ be the
components of $T_n-v$. Since $T_n$ is a tree, it follows that
$|E_G[v,V(C_i)]|=1$ for each $i \ (1\leq i\leq r)$. Let $x$ be the
number of isolated vertices in $T_n-v$. Since $|X|=n-t$ and $v\in
X$, it follows that $x\leq n-t-1$. Furthermore, we have the
following claim.

\begin{claim}\label{Clm:2}
$x=n-t-1$.
\end{claim}
\begin{proof}
Assume, to the contrary, that $x\leq n-t-2$. By deleting these $x$
isolated vertices and $v$, the minimum degree of each component of
the resulting graph is at least $1$, and hence $\kappa^1(T_n)\leq
x+1\leq n-t-2+1=n-t-1<n-t$, a contradiction.
\end{proof}

Let $C_i \ (1\leq i\leq n-t-1)$ be the isolated vertices in $T_n-v$.
Then we have the following claim.
\begin{claim}\label{Clm:3}
For each $i \ (n-t\leq i\leq r)$, $C_i$ is a star.
\end{claim}
\begin{proof}
Assume, to the contrary, that there exists some $C_j$ such that
$C_j$ is not a star. Then $C_j$ contains a $2K_2$, say
$u_1u_2,u_3u_4$. Let $vv_j$ be the unique edge from $v$ to $C_j$.
Let $W$ be the set of pendent vertices adjacent to $v_j$ in $C_j$.
Since $C_j$ is not a star, it follows that there is at least one
edge in $C_j-v_j-W$. Since $T_n-X$ is not connected, it follows that
except $C_j$, there exists another component of order at least $2$,
say $C_k$. Then
$$
\kappa^1(T_n)\leq |\{v_j\}\cup W|\leq n-(n-t)-2-2\leq t-4<n-t,
$$
a contradiction.
\end{proof}

From Claim \ref{Clm:3}, $C_i$ is a star for each $i \ (n-t\leq i\leq
r)$. Let $w_i$ be the center of $C_i$, where $n-t\leq i\leq r$. Then
we have the following claim.

\begin{claim}\label{Clm:4}
For each $i \ (n-t\leq i\leq r)$, we have $vw_i\in E(T_n)$.
\end{claim}
\begin{proof}
Assume, to the contrary, that there exists some $w_j$ such that
$vw_j\notin E(T_n)$. Then there exists a vertex $a_j$ such that
$a_jv\in E(T_n)$. Note that $a_j$ is a leaf of $C_j$. Then $T_n-a_j$
is not connected and the minimum degree of each component is at
least $1$, and hence $\kappa^1(T_n)\leq 1$, a contradiction.
\end{proof}

From Claim \ref{Clm:4}, $vw_i\in E(T_n)$ for each $i \ (n-t\leq
i\leq r)$. Then $T_n=T_n^*$.
\end{proof}

\subsection{Graphs with small $g$-good-neighbor connectivity}

Graphs with $\kappa^g(G)=1$ can be characterized easily.
\begin{observation}\label{obs4-1}
Let $n,g$ be two integers and let $G$ be a connected graph of order
$n$ with $0\leq g\leq \left\{\Delta(T_n),\left\lfloor
\frac{n-3}{2}\right\rfloor\right\}$. Then $\kappa^g(G)=1$ if and
only if there exists a cut vertex $v$ in $G$ such that the minimum
degree of each connected component of $G-v$ is at least $g$.
\end{observation}

We can also characterize graphs with $\kappa^g(G)=2$.
\begin{theorem}\label{th3-2}
Let $n,g$ be two integers and let $G$ be a connected graph of order
$n$ with $0\leq g\leq \left\{\Delta(T_n),\left\lfloor
\frac{n-3}{2}\right\rfloor\right\}$. Then $\kappa^g(G)=2$ if and
only if $G$ satisfies one of the following conditions.

$(1)$ $\kappa(G)=2$ and there exists a cut vertex set $\{u,v\}$ in
$G$ such that the minimum degree of each component of $G-\{u,v\}$ is
at least $g$;

$(2)$ $\kappa(G)=1$, and $g\geq 1$, and $(2.1),(2.2)$ hold, where
\begin{itemize}
\item[] $(2.1)$ for each cut vertex $u$, there exists a component
of $G-u$ such that its minimum degree is at most $g-1$,

\item[] $(2.2)$ $(a)$ there exists a cut vertex $v$ such that there is exactly one component in
$G-v$ having exactly one vertex $u$ of degree at most $g-1$ and the
neighbors of $u$ has degree at least $g+1$ and the minimum degree of
other vertices is at least $g$, or $(b)$ there exists a cut vertex
$v$ such that $G-v$ contains at least $3$ components, where one of
the component is an isolated vertex and the minimum degree of each
of the other components is at least $g$, or $(c)$ there are two
non-cut vertices $x,y$ such that $G-\{x,y\}$ is not connected and
the minimum degree of each component is at least $g$.
\end{itemize}
\end{theorem}
\begin{proof}
Suppose that $G$ satisfies $(1)$ and $(2)$. Suppose that $(1)$
holds. Since the minimum degree of each component of $G-\{u,v\}$ is
at least $g$, it follows that $\kappa^g(G)\leq 2$. From Proposition
\ref{pro3-1}, we have $\kappa^g(G)\geq \kappa(G)=2$.

Suppose that $(2)$ holds. Since for each cut vertex $u$, there
exists a component of $G-u$ such that its minimum degree is at most
$g-1$, it follows that $\kappa^g(G)\geq 2$. If there exists a cut
vertex $v$ such that there is exactly one component in $G-v$ having
exactly one vertex $u$ of degree at most $g-1$ and the neighbors of
$u$ has degree at least $g+1$ and the minimum degree of other
vertices is at least $g$, then $\kappa^g(G)\leq |\{u,v\}|=2$. If
there exists a cut vertex $v$ such that $G-v$ contains at least $3$
components, where one of the component is an isolated vertex $u$ and
the minimum degree of each of the other components is at least $g$,
then $\kappa^g(G)\leq |\{u,v\}|=2$. If there are two non-cut
vertices $x,y$ such that $G-\{x,y\}$ is not connected and the
minimum degree of each component is at least $g$, $\kappa^g(G)\leq
|\{x,y\}|=2$. So we have $\kappa^g(G)=2$.

Conversely, we suppose $\kappa^g(G)=2$. From Proposition
\ref{pro3-1}, we have $\kappa(G)\leq 2$. Suppose $\kappa(G)=2$. If
for each vertex cut set $\{u,v\}$ in $G$, there exists a component
of $G-\{u,v\}$ such that its minimum degree is at most $g-1$, then
$\kappa^g(G)\geq 3$, a contradiction. So there exists a vertex cut
set $\{u,v\}$ in $G$ such that the minimum degree of each component
of $G-\{u,v\}$ is at least $g$, as desired.

Suppose $\kappa(G)=1$. Then we have the following claim.

\begin{claim}\label{Clm:5}
$g\geq 1$.
\end{claim}
\begin{proof}
Assume, to the contrary, that $g=0$. Since $\kappa(G)=1$, it follows
that there exists a cut vertex $v$ and the minimum degree of each
component of $G-v$ is at least $0$, and hence $\kappa^0(G)\leq 1$,
which contradicts to the fact $\kappa^g(G)=2$.
\end{proof}

From Claim \ref{Clm:5}, we have $g\geq 1$. Since $\kappa^g(G)=2$, we
have the following facts.
\begin{fact}\label{Fact:Fact2}
For any cut vertex $v$, there exists a component of $G-v$ such that
its minimum degree is at most $g-1$.
\end{fact}

\begin{fact}\label{Fact:Fact3}
There exist two vertices $x,y$ in $G$ such that $G-\{x,y\}$ is not
connected and the minimum degree of each component of $G-\{x,y\}$ is
at least $g$.
\end{fact}

Suppose that one of $x,y$ is a cut vertex of $G$. Without loss of
generality, we assume that $x$ is a cut vertex of $G$. Let
$C_1,C_2,\ldots,C_r$ be the components of $G-x$.

\begin{claim}\label{Clm:6}
At most one of $C_1,C_2,\ldots,C_r$ has exactly one vertex.
\end{claim}
\begin{proof}
Assume, to the contrary, that there exist $C_i,C_j \ (1\leq i\neq
j\leq r)$ such that $|V(C_i)|=|V(C_j)|=1$. Then at least one
component of $G-x-y$ is a isolated vertex, which contradicts to Fact
\ref{Fact:Fact3}.
\end{proof}

From Claim \ref{Clm:6}, if one of $C_1,C_2,\ldots,C_r$, say $C_1$,
has exactly one vertex, then $C_1=\{y\}$ and $(2.2) (b)$ holds.
Suppose that each $C_i \ (1\leq i\leq r)$ has at least two vertices.
\begin{claim}\label{Clm:7}
Exactly one of $C_1,C_2,\ldots,C_r$ has minimum degree at most
$g-1$.
\end{claim}
\begin{proof}
Assume, to the contrary, that there exist $C_i,C_j \ (1\leq i\neq
j\leq r)$ such that $\delta(C_i)\leq g-1$ and $\delta(C_j)\leq g-1$.
Then there is a component of $G-x-y$ such that its minimum degree is
at most $g-1$, which contradicts to Fact $2$.
\end{proof}

From Claim \ref{Clm:7}, exactly one of $C_1,C_2,\ldots,C_r$, say
$C_1$, has minimum degree at most $g-1$. Then there exists a vertex
$u$ of degree at most $g-1$. We claim that $u=y$. Assume, to the
contrary, that the degree of $u$ is most $g-1$ in $G-x-y$, a
contradiction. Then $u=y$. From Fact \ref{Fact:Fact2}, $(2.2) (a)$
holds.

Suppose that neither $x$ nor $y$ is a cut vertex of $G$. From Fact
\ref{Fact:Fact2}, $G-\{x,y\}$ is not connected and the minimum
degree of each component is at least $g$. Then $(2.2) (c)$ holds.
\end{proof}

For $(1)$, $(2.2)(a)$, $(2.2)(b)$ and $(2.2)(c)$, we show the
following examples corresponding them.\\

\noindent {\bf Example 4.1.} Let $H_1$ be a graph obtained from
$K_{g+1}$ and $K_{n-g-3}$ by adding two new vertices $u,v$ and edges
in $E_{H_1}[u,V(K_{g+1})]\cup E_{H_1}[u,V(K_{n-g-3})]\cup
E_{H_1}[v,V(K_{g+1})]\cup E_{H_1}[v,V(K_{n-g-3})]$, where $n\geq
2g+4$. From Theorem \ref{th3-2},
$\kappa(H_1)=\kappa_g(H_1)=2$.\\

\noindent {\bf Example 4.2.} Let $H_2$ be a graph obtained from
$K_{g+1}$ and $K_{n-g-3}$ by adding two new vertices $u,v$ and edges
in $\{uv,vv_1,uu_1\}$, where $u_1\in V(K_{n-g-3})$ and $v_1\in
V(K_{g+1})$ and $n\geq 2g+4$. From Theorem \ref{th3-2},
$\kappa(H_2)=1$ and $\kappa_g(H_2)=2$.\\

\noindent {\bf Example 4.3.} Let $H_3$ be a graph obtained from
$K_{g+1}$ and $K_{n-g-3}$ by adding two new vertices $u,v$ and edges
in $\{uv,vv_1,vv_2\}$, where $v_1\in V(K_{n-g-3})$ and $v_2\in
V(K_{g+1})$ and $n\geq 2g+4$. From Theorem \ref{th3-2},
$\kappa(H_3)=1$ and $\kappa_g(H_3)=2$.\\

\noindent {\bf Example 4.4.} Let $H_4$ be a graph obtained from
$K_{g+1}$ and $K_{n-g-1}$ by adding edges in $\{vv_1,vv_2\}$, where
$v\in V(K_{g+1})$ and $v_1,v_2\in V(K_{n-g-1})$ and $n\geq 2g+4$.
From Theorem \ref{th3-2}, $\kappa(H_4)=1$ and $\kappa_g(H_4)=2$.

\section{Extremal problems}

We now consider the three extremal problems that we stated in the
Introduction.

Suppose that $n,k,g$ are three integers such that $(n-k)g$ is even
and $2\leq g\leq \left\lfloor \frac{n-k-2}{2}\right\rfloor$. Let
$F_1,F_2$ be two $g$-regular graphs such that
$|V(F_1)|+|V(F_2)|=n-k$. Let $F^k_n$ be a graph obtained from
$F_1,F_2$ and a star $K_{1,k-1}$ with center $v$ such that
$|E_G[v,F_1]|=|E_G[v,F_1]|=1$.
\begin{lemma}\label{lem4-1}
Let $n,g,k$ be three integers with $2\leq g\leq \left\lfloor
\frac{n-k-2}{2}\right\rfloor$. If $(n-k)g$ is even, then
$$
\kappa^g(F^k_n)=k.
$$
\end{lemma}
\begin{proof}
Let $X=V(K_{1,k-1})$. Then $F^k_n-X$ is not connected and the
minimum degree of each component of $G-X$ is at least $g$, and hence
$\kappa^g(F^k_n)\leq k$. Let $\kappa^g(F^k_n)=t$. It suffices to
show $t\geq k$. From the definition of $\kappa^g(F^k_n)$, there
exists $X\subseteq V(F^k_n)$ with $|X|=t$ such that the minimum
degree of each component of $F^k_n-X$ is at least $g$. Since $g\geq
2$, it follows that $V(K_{1,k-1})-v\subseteq X$. If
$X=V(K_{1,k-1})-v$, then $G$ is connected, a contradiction. So
$|X|\geq |V(K_{1,k-1})|-1+1=k$. So $\kappa^g(F^k_n)\geq k$, and
hence $\kappa^g(F^k_n)=k$.
\end{proof}

Suppose that $n,k,g$ are three integers such that $(n-k)g$ is odd
and $2\leq g\leq \left\lfloor \frac{n-k-2}{2}\right\rfloor$. Then
$n-k$ is odd. Let $a,b$ be two integers such that $a$ is even, and
$b$ is odd, and $a\geq g+1$, and $b\geq g+1$, and $a+b=n-k$. Let
$H_1$ be a $g$-regular graph of order $a$. Let $H_2$ be a graph of
order $b$ such that the degree of one vertex is exactly $g+1$, and
the degree of each of the other vertices is exactly $g$. Let $H^k_n$
be a graph obtained from $H_1,H_2$ and a star $K_{1,k-1}$ with
center $v$ such that $|E_G[v,H_1]|=|E_G[v,G_1]|=1$. Similarly, we
have the following lemma.
\begin{lemma}\label{lem4-2}
Let $n,g,k$ be three integers with $2\leq g\leq \left\lfloor
\frac{n-k-2}{2}\right\rfloor$. Then
$$
\kappa^g(H^k_n)=k.
$$
\end{lemma}

Let $T_n'$ be a tree of order $n$ obtained from three stars
$K_{1,k-1},K_{1,a-1},K_{1,b-1}$ with centers $x,u,v$ by adding two
edges $xu,xv$.
\begin{lemma}\label{lem4-3}
Let $n,k$ be two integers with $n\geq k+4$. Then
$$
\kappa^1(T_n')=k.
$$
\end{lemma}
\begin{proof}
Let $X=V(K_{1,k-1})$. Then $T_n'-X$ is not connected and the minimum
degree of each component of $T_n'-X$ is at least $1$, and hence
$\kappa^1(T_n')\leq k$. Let $\kappa^1(T_n')=t$. It suffices to show
$t\geq k$. From the definition of $\kappa^1(T_n')$, there exists
$X\subseteq V(T_n')$ with $|X|=t$ such that the minimum degree of
each component of $T_n'-X$ is at least $1$. If $x\notin X$, then
$u\in X$ or $v\in X$. Without loss of generality, let $u\in X$. Then
$V(K_{1,a-1})\subseteq X$. Since $T_n'-X$ is not connected and
$x\notin X$, it follows that $v\in X$. Then $V(K_{1,b-1})\subseteq
X$. Clearly, $T_n'-X$ is connected, a contradiction. If $x\in X$,
then $V(K_{1,k-1})\subseteq X$, and hence $|X|=t\geq k$. So
$\kappa^1(T_n')\geq k$, and hence $\kappa^1(T_n')=k$.
\end{proof}

\begin{theorem}\label{pro5-1}
Let $n,g,k$ be three integers with $1\leq g\leq \left\lfloor
\frac{n-k-2}{2}\right\rfloor$ and $1\leq k\leq n-2g-2$.

$(1)$ If $g=1$, then $s(n,k)=n-1$.

$(2)$ If $n-k$ is even and $g\geq 2$, then
$$
s(n,k)=\frac{(n-k)g}{2}+k+1.
$$

$(3)$ If $n-k$ is odd and $g\geq 2$, then
$$
s(n,k)=\frac{(n-k)g+1}{2}+k+1.
$$
\end{theorem}
\begin{proof}
$(1)$ Let $G=T_n'$. From Lemma \ref{lem4-3}, we have $s(n,k)\leq
n-1$. Since we only consider connected graphs, it follows that
$s(n,k)\geq n-1$, and hence $s(n,k)=n-1$.

$(2)$ Suppose that $n-k$ is even. Let $G=F^k_n$. From Lemma
\ref{lem4-1}, we have $s(n,k)\leq \frac{(n-k)g}{2}+k+1$. It suffice
to show $s(n,k)\geq \frac{(n-k)g}{2}+k+1$. Let $G$ be a conneced
graph of order $n$ with $\kappa_g(G)=k$ such that $e(G)$ is
minimized. Then exists $X\subseteq V(G)$ with $|X|=k$ such that the
minimum degree of each component of $G-X$ is at least $g$. Then
$e(G-X)\geq \frac{(n-k)g}{2}$. Since $G$ is connected, it follows
that $e(G)\geq \frac{(n-k)g}{2}+k+1$, and hence
$s(n,k)=\frac{(n-k)g}{2}+k+1$.

$(3)$ Suppose that $n-k$ is odd. Let $G=H^k_n$. From Lemma
\ref{lem4-2}, we have $s(n,k)\leq \frac{(n-k)g+1}{2}+k+1$. Similarly
to the proof of $(1)$, we have $s(n,k)=\frac{(n-k)g+1}{2}+k+1$.
\end{proof}

\begin{lemma}\label{lem5-2}
Let $n,g,k$ be three integers with $1\leq g\leq \left\lfloor
\frac{n-k-2}{2}\right\rfloor$. Let $G^k_n$ be the graph obtained
from three cliques $K_{n-k-g},K_{k-1},K_{g+1}$ by adding the edges
in $E_{G^k_n}[K_{n-k-g},K_{k-1}]\cup E_{G^k_n}[K_{g+1},K_{k-1}]$.
Then
$$
\kappa^g(G^k_n)=k-1.
$$
\end{lemma}
\begin{proof}
Let $X=K_{k-1}$. Since $1\leq g\leq \left\lfloor
\frac{n-k-2}{2}\right\rfloor$, it follows that $G^k_n-X$ is not
connected and each component has at least $g+1$ vertices, and hence
$\kappa^g(G^k_n)\leq k-1$. Clearly, $\kappa^g(G^k_n)\geq
\kappa(G^k_n)=k-1$, and hence $\kappa^g(G^k_n)=k-1$.
\end{proof}

\begin{theorem}\label{th5-2}
Let $n,g,k$ be two integers with $1\leq g\leq \left\lfloor
\frac{n-k-2}{2}\right\rfloor$ and $1\leq k\leq n-2g-2$. Then
$$
f(n,k)={n\choose 2}-(n-k-g)(g+1)+1.
$$
\end{theorem}
\begin{proof}
To show $f(n,k)\geq {n\choose 2}-(n-k-g)(g+1)+1$, we construct
$G^k_n$ defined in Lemma \ref{lem5-2}. Then $\kappa^g(G^k_n)=k-1$.
Since $|E(G^k_n)|={n\choose 2}-(n-k-g)(g+1)$, it follows that
$f(n,k)\geq {n\choose 2}-(n-k-g)(g+1)+1$.

Let $G$ be a graph with $n$ vertices such that $|E(G)|\geq {n\choose
2}-(n-k-g)(g+1)+1$. We claim that $\kappa^g(G)\geq k$. Assume, to
the contrary, that $\kappa^g(G)\leq k-1$. Then there exists a vertex
set $X\subseteq V(G)$ and $|X|\leq k-1$ such that the minimum degree
of each component of $G-X$ is at least $g$. Let $C_1,C_2,\ldots,C_r$
be the components of $G-X$. The number of edges from $C_1$ to
$C_2\cup C_3\cup \ldots \cup C_r$ in $\overline{G}$ is at least
$|V(C_1)|(n-|V(C_1)|-|X|)\geq (n-k-g)(g+1)$ since
$\sum_{i=1}^r|V(C_i)|\geq n-k+1$ and $|V(C_i)|\geq g+1$ for each $i
\ (1\leq i\leq r)$. Clearly, $|E(G)|\leq {n\choose 2}-(n-k-g)(g+1)$,
which contradicts to $|E(G)|\geq {n\choose 2}-(n-k-g)(g+1)+1$. So
$\kappa^g(G)\geq k$, and hence $f(n,k)\leq {n\choose
2}-(n-k-g)(g+1)+1$.

From the above argument, we have $f(n,k)={n\choose
2}-(n-k-g)(g+1)+1$.
\end{proof}

Note that $g(n,k)=s(n,k+1)-1$. So we have the following.
\begin{proposition}\label{pro5-2}
Let $n,g,k$ be three integers with $2\leq g\leq \left\lfloor
\frac{n-k-2}{2}\right\rfloor$ and $1\leq k\leq n-2g-2$.

$(1)$ If $n-k$ is odd and $g\geq 2$, then
$g(n,k)=\frac{(n-k-1)g}{2}+k+1$.

$(2)$ If $n-k$ is even and $g\geq 2$, then
$g(n,k)=\frac{(n-k-1)g+1}{2}+k+1$.
\end{proposition}

\section{Concluding Remark}

In this paper, we focus our attention on the $g$-good neighbor
connectivity of general graphs. We have proved that $1\leq
\kappa^g(G)\leq n-2g-2$ for $0\leq g\leq
\min\{\Delta(G),\left\lfloor \frac{n-3}{2}\right\rfloor\}$. Trees
with $\kappa^g(T_n)=n-t$ are characterized in this paper. But the
graphs with $\kappa^g(G)=n-t$ is still unknown. From Proposition
\ref{pro3-1}, the classical $\kappa(G)$ is a natural lower bound of
$\kappa^g(G)$, but there is no upper bound of $\kappa^g(G)$ in terms
of $\kappa(G)$.


\end{document}